\documentclass{amsart}

\usepackage[utf8]{inputenc}

\usepackage{amsmath}
\usepackage{amsfonts}
\usepackage{microtype}
\usepackage{bbm}
\usepackage{amsthm}
\usepackage{graphicx}
\usepackage{xspace}
\usepackage{tikz}
\usepackage{caption}
\usepackage{subfig}

\usepackage[
text={440pt,575pt},
headheight=9pt,
centering
]{geometry}
\usepackage[autostyle]{csquotes}

\allowdisplaybreaks

\usepackage{xcolor}
\usepackage{hyperref}
\usepackage{url, hypcap}
\definecolor{darkblue}{RGB}{0,0,160}
\hypersetup{
	colorlinks,%
	citecolor=darkblue,%
	filecolor=black,%
	linkcolor=darkblue,%
	urlcolor=darkblue
}

\newcommand{\excise}[1]{}

\newcommand{\RR}{\mathbb{R}}
\newcommand{\NN}{\mathbb{N}}

\newtheorem{thm}{Theorem}
\newtheorem{lemma}[thm]{Lemma}

\newtheorem{cor}[thm]{Corollary}
\newtheorem{prop}[thm]{Proposition}
\newtheorem{conj}[thm]{Conjecture}

\newtheorem{prob}[thm]{Problem}

\theoremstyle{definition}
\newtheorem{example}[thm]{Example}
\newtheorem{remark}[thm]{Remark}
\newtheorem{defn}[thm]{Definition}

\numberwithin{equation}{section}
\newcommand{\Mathematica}{\textsc{Mathematica}\xspace}
\newcommand{\Mtwo}{\textsc{Macaulay2}\xspace}
\begin{document}

	\title[The semi-algebraic geometry of optimal designs for the
	Bradley--Terry model]{The semi-algebraic geometry of saturated optimal designs for the
		Bradley--Terry model}
	
	\author[T.~Kahle]{Thomas Kahle}
	\address[]{Fakultät für Mathematik\\Otto-von-Guericke Universität
		Magdeburg\\39106 Magdeburg\\Germany}
	\email{thomas.kahle@ovgu.de}
	\urladdr{\url{http://www.thomas-kahle.de}}
	\author[F.~Röttger]{Frank Röttger}
	\address[]{Research Center for Statistics\\Université de Genève\\Boulevard du Pont d’Arve 40\\ 1205 Genève\\Switzerland}
	\email{frank.roettger@unige.ch}
	\urladdr{\url{https://sites.google.com/view/roettger/}}
	\author[R.~Schwabe]{Rainer Schwabe}
	\address[]{Fakultät für Mathematik\\Otto-von-Guericke Universität
		Magdeburg\\39106 Magdeburg\\Germany}
	\email{rainer.schwabe@ovgu.de}
	\urladdr{\url{http://www.imst3.ovgu.de}}
	
	\makeatletter
	\@namedef{subjclassname@2010}{\textup{2020} Mathematics Subject Classification}
	\makeatother
	
	\subjclass[2010]{Primary: 62K05, 62R01 Secondary: 13P25, 14P10, 62J02}
	

	\begin{abstract}
          Optimal design theory for nonlinear regression studies local
          optimality on a given design space.  We identify designs for
          the Bradley--Terry paired comparison model with small
          undirected graphs and prove that every saturated, locally
          $D$-optimal design is represented by a path.  We discuss the
          case of four alternatives in detail and derive explicit
          polynomial inequality descriptions for optimality regions in
          parameter space.  Using these regions, for each point in
          parameter space we can prescribe a locally $D$-optimal design.
		
          \vspace{5pt}
          \noindent {\it Key words and phrases:}
          nonlinear regression, optimal design, polynomial
          inequalities
	\end{abstract}
	
	\maketitle

\section{Introduction}

Consider an experimental situation in which $ m $ alternatives are to be
brought into a rank order. For each single observation in this
experiment only two of these alternatives can be compared at a time
and only a binary response can be observed which indicates the rank
order of the two alternatives presented. Such experiments are known in
economics as ``discrete choice'' experiments and in psychology as
``forced choice'' experiments or ``ipsative measures''. The use of
such experiments dates back to the work by Fechner~\cite{Fechner1860} on
psychophysics in a deterministic setup. In a statistical setup this
situation is described by the Bradley-Terry model, which was introduced
in \cite{zermelo29} to rank chess players in tournaments and in
\cite{BradleyTerry} to analyze taste testing results for pork
depending on different feeding patterns.  See \cite{kahle20:DMV} for a
leisurely introduction.  This model has proven popular in different
areas of statistics, also outside of chess tournaments and pork
tasting. In \cite{Hastie1998}, Hastie and Tibshirani developed a coupling model similar to the
Bradley--Terry model to study class probabilities for pairs of
classes.  \cite{SimonsYao1999} discussed the model asymptotics when
the number of potential alternatives tends to infinity.  Algorithms
for Bradley--Terry models are discussed, for example,
in~\cite{Hunter2004}, and asymptotics of algorithms, for example,
in~\cite{DuchiMackeyJordan2013}. Besides marketing or transportation,
another popular application area for the Bradley--Terry model is the
world of professional sports such as American football, car racing,
matching in tournaments, card games or strategies for sport bets,
see~\cite{Callaghan2007, Graves2003,Berman2004,Kobayashi2006}.  The
Bradley--Terry model is part of a broader class of models that
describe statistical rankings.  Specifically, it arises from
marginalization of the Plackett--Luce model,
see~\cite{sturmfelswelker}.

In this paper we are interested in optimal experimental designs for
the Bradley--Terry model, that is, a scheme to assign a fixed
number of measurements to different experimental settings, such that
the experiment is most informative about the parameters.  Optimal
experimental designs for the Bradley--Terry model were first
investigated in \cite{Torsney2004}, which gave an algorithmic
approach to fit the model parameters. In \cite{Grasshoff2008} Gra{\ss}hoff and Schwabe
completely analyzed the case of three competing alternatives (with pair comparisons).  They gave
symbolic solutions for the design problem depending on the parameters
and described the optimality areas of these design classes in the
parameter space.
The present paper extends the results of Gra{\ss}hoff and Schwabe in
two directions.  We discuss the case of four competing alternatives in detail
and characterize optimal saturated designs for an arbitrary number of
competing alternatives, always with pair comparisons.
The case of four alternatives arises also when considering a $ 2^2 $ layout with interaction where two attributes can be set to two levels each. After a reparametrization, which does not affect the $D$-optimality, this model can be identified as a single-attribute model with four levels which can be used as alternatives in the Bradley--Terry model.

Section~\ref{s:generalSetup} gives the general
setup. Section~\ref{s:fouralternatives} contains an almost complete
analysis of the case of 4 alternatives.  Only one very challenging
polynomial inequality system remains open
(Problem~\ref{prob:infeasible}).  In
Sections~\ref{s:graphrepresentation} and~\ref{s:optregion} we discuss
saturated optimal designs for an arbitrary number of alternatives.
Our main result is an easy combinatorial polynomial inequality
description of regions in parameter space where a given saturated
design is optimal, including the information for which designs these
regions of optimality are empty (Theorem~\ref{t:regions}).  Polynomial
inequality constraints in experimental design are a recurrent topic.
See \cite{kahle2015algebraic} for a discussion of this principle for
Poisson regression.
Knowledge about the optimality regions can be very helpful in
designing experiments.  For example, a screening experiment could
reveal that the estimates of the parameters are all within one region
of optimality.  In such a situation it is then clear which design to
use.  See \cite{dette2004optimal} for a general class of models where
local optimality is studied.  In the discussion in
Section~\ref{sec:discussion} we compare the efficiency of our tailored
designs versus uniform designs as the parameters grow in magnitude.

\section{General Setup}\label{s:generalSetup}
We consider pairs $(i,j)$ of alternatives $i,j=1,\ldots,m$.  The
preference of $i$ over $j$ is modeled by a binary variable $ Y(i,j) $
taking the value $ Y(i,j)=1 $ if $ i $ is preferred over $ j $ and
$ Y(i,j)=0 $ otherwise.  We do not consider any order effects here.
The main assumption of the Bradley--Terry model is that there is a
hidden ranking of the alternatives according to some numerical
preference value $\pi_i>0$, $i=1,\dots,m$.  When presented with the
pair $(i,j)$, the probability of preferring $i$ over $j$ is
\[
\mathbb{P}( Y(i,j)=1 )=\frac{\pi_i}{\pi_i+\pi_j}.
\]
The model can be transformed into a logistic model using
$ \beta_i:=\log(\pi_i)$.  Then
\[
\mathbb{P}(Y(i,j)=1)=\frac{1}{1+\exp(-(\beta_i-\beta_j))}=\eta(\beta_i-\beta_j)
\]
with $ \eta(z)=(1+\exp(-z))^{-1} $ as the inverse logit link function.

Scaling all $\pi_i$ with a constant factor leaves the preference
probabilities invariant.  Therefore one can without loss of generality
assume that $\pi_m=1$ or $\beta_m=0$. This means that the number of
parameters of the Bradley--Terry model is~$m-1$.  The number of
alternatives is the main measure of complexity of the design theory as
it equals the dimension of the design space. The remaining parameters
can be identified and $\beta_m=0$ is known as \emph{control coding}.
We denote by $e_i$ the $i$-th standard unit vector in $\RR^{m-1}$.  To
exhibit our model as a generalized linear model, the \emph{regression
  vector} for a pair $(i,j)$ is
\[
f(i,j)=
\begin{cases} e_i-e_j, & \text{for } i,j\neq m,\\
  e_i, & \text{for } i<j, j=m,\\
  0 & \text {for } i=j=m.
  \end{cases}
\]
With $ \beta^{T}=(\beta_1,\ldots,\beta_{m-1}) $ this yields
$\mathbb{P}(Y(i,j)=1)=\eta(f(i,j)^{T}\beta)$ where $f(i,j)^T \beta$ is
the linear predictor.

\begin{remark}
  When all probabilities
   \[p_{ij}:=\mathbb{P}(Y(i,j)=1)=\frac{\pi_i}{\pi_i+\pi_j}~~\text{ for }
  i,j\in[m]:=\{1,2,\ldots,m\}\]  are treated as coordinates
  in~$\RR^{m(m-1)}$, the Bradley--Terry model can be described by
  algebraic equations.  This means that all values of the $p_{ij}$
  that arise for different values of $\pi$ satisfy certain algebraic
  equations and, among the probability vectors, they are the only
  solutions to these equations. Theorem~7.7 of \cite{sturmfelswelker}
  shows that the model has the special geometric structure of a toric
  variety and its defining equations consist of binomials and linear
  trinomials.
\end{remark}
The \emph{design region} of the Bradley--Terry paired comparison model
is
\[
\mathcal{X} = \{(i,j)\, :\, i,j=1,...,m,\, i<j\}.
\]
It consists of all pairs of ordered alternatives. The pairs $(i,j)$
and $(j,i)$ bear the same information, and the comparison $(i,i)$ of
two identical alternatives does not have any information at all (as
can be seen easily later).  Therefore, whenever there are two
alternatives $i,j\in\{1,\dots,m\}$ we assume~$i<j$.  An
\emph{experimental design} is an assignment of a weight $w_{ij} \ge 0$
to each point $(i,j) \in \mathcal{X}$, such that
$\sum_{ij} w_{ij} = 1$ (compare, for example,
\cite{silvey1980optimal}).  Although a design could be impossible to
realize with a finite number $N$ of observations, it is common to let
$w_{ij} \in \RR$ as opposed to $w_{ij}\in\frac{1}{N}\NN$.
For any $k\in \NN$ we write
\[
\Delta_k := \{w \in \RR^k_{\ge 0} : \sum_l w_l = 1\},
\]
for the $k-1$ dimensional simplex in $\RR^k$, whose vertices are the
standard unit vectors.  It is customary to use $\xi$ to refer to
a design with weights $w_{ij}$ and slightly abuse notation with
expressions like $\xi \in \Delta_{\binom{m}{2}}$.

The information gained from one observation of $Y(i,j)$ is encoded in
the information matrix
\[
M((i,j),\beta)=\lambda_{ij}f(i,j)f(i,j)^{T} \in \RR^{(m-1)\times (m-1)},
\]
where
\[
\lambda_{ij}:=\lambda_{i,j}(\beta)=\eta'(\beta_i-\beta_j)=\frac{e^{\beta_i-\beta_j}}{(1+e^{\beta_i-\beta_j})^2}
\]
is referred to as the \emph{intensity} in~\cite{Grasshoff2008}.  It
holds that $ M((i,j),\beta)=M((j,i),\beta) $ and $ M((i,i),\beta)=0 $.

Assuming independent observations, the information matrix for a design $\xi$ with weights $w_{ij}$ is the
$ (m-1)\times(m-1)$-matrix
\begin{equation}\label{eq:Informationsmatrix}
M(\xi,\beta)=\sum_{(i,j)}w_{ij}M((i,j),\beta)=\sum_{(i,j)}w_{ij}\lambda_{ij}f(i,j)f(i,j)^{T}.
\end{equation}
The theory of optimal experimental design suggests picking weights
$w_{ij}$ that optimize a numerical function of~$M(\xi,\beta)$.
Standard references that include the theory for generalized linear
models are \cite{pukelsheim2006optimal,silvey1980optimal}.  One
popular function to optimize is the logarithm of the determinant:
\begin{defn}
An experimental design $\xi^*$ is locally \emph{$ D $-optimal}, if 
\[\log\det(M(\xi^*,\beta))\ge\log\det(M(\xi,\beta))\] 
for all~$\xi \in \Delta_{{\binom{m}{2}}}$.
\end{defn}
In optimal experimental design one speaks of local optimality if the
optimal choice of a design depends on the unknown parameters that one
wants to learn about, see ~\cite{Chernoff}.  From the perspective of
mathematical optimization one has a parametric family of convex
optimization problems where both the optimization domain (the polytope
of information matrices) and the target function depend on the
parameters~$\beta$.  The methods of convex optimization suggest studying the directional derivatives of the target function.  The
following is found in \cite[Section 3.5.2]{silvey1980optimal}.
\begin{defn} The \emph{directional derivative} (\emph{Fr\'echet
    derivative}) of the $ D $-optimality criterion at $ M_1 $ in the
  direction of $ M_2 $ for some $(m-1) \times (m-1)$ information matrices
  $ M_1,M_2 $ is
\[
F_D(M_1,M_2)=\lim_{\varepsilon\searrow
0}\frac{1}{\varepsilon}\left(\log\det((1-\varepsilon)M_1+\varepsilon
  M_2)-\log \det(M_1) \right).
\]
\end{defn}
It is shown in \cite[Sections~3.8 and~3.11]{silvey1980optimal} that
\begin{align}\label{eq:directionalder}
F_D(M(\xi,\beta),M((i,j),\beta))=\lambda_{ij}f(i,j)^T M(\xi,\beta)^{-1} f(i,j)- (m-1).
\end{align}
This yields the following $ D $-optimality criterion:

\begin{thm}[Kiefer--Wolfowitz] \label{t:silvey}
A design $ \xi^* $ is locally $ D $-optimal if and only if 
\begin{align}\label{eq:optimalityregions}
\lambda_{ij}f(i,j)^T M(\xi^*,\beta)^{-1} f(i,j)\le m-1
\end{align}
for all $ 1\le i<j\le m $ .
\end{thm}
The following corollary from \cite[Corollary~3.10]{silvey1980optimal}
is very useful.
\begin{cor}\label{c:silvey}
For design points $(i,j)$ with positive weight in $ \xi^{*} $, the inequalities
\eqref{eq:optimalityregions} in Theorem~\ref{t:silvey} hold with
equality.
\end{cor}

A main observation about the Bradley--Terry model is that it is useful
to represent pairs $(i,j)$ with positive weights $w_{ij}$ as the edges
of an undirected graph on the vertex set $\{1,\dots,m\}$.  Properties
of these graphs, in particular the edge density, determine the
asymptotics of estimation for sparse Bradley--Terry
models~\cite{han2020asymptotic}.

\begin{defn}\label{d:graphrepresentation}
A \emph{graph representation of a design $\xi$} for the Bradley--Terry
model is the undirected simple graph with vertex set $\{1,\dots,m\}$,
and edge set $E = \{(i,j): w_{ij} > 0\}$.
\end{defn}

Using standard notions from graph theory, a \emph{tree} is a connected
graph with no cycles.  A \emph{path} is a tree in which every vertex
is connected to at most two other vertices.
We exploit the symmetry of the model.  The symmetric group $S_{m}$ of
all bijective self-maps of $\{1,\dots,m\}$ permutes the alternatives.
The permutation action extends to ordered pairs by acting on both
entries of the pair simultaneously (and changing the order if
necessary).  The action also extends naturally to designs $\xi$ on
pairs $(i,j)$ by putting
$(\xi^\sigma)_{(i,j)} = \xi_{\sigma^{-1}(i,j)}$ for any $\sigma\in S_m$. 
A graph representation of an entire orbit under this action is simply the
unlabeled graph.  Proposition~\ref{p:symmetric} below expresses that
for properties of the model it is irrelevant which alternative is
alternative~1, which is alternative 2 and so on.  One only needs to
take care that upon relabeling the parameters, regression vectors,
etc.\ are relabeled accordingly.

In our setup we have singled out the last alternative~$m$ and set
$\beta_m = 0$ to have identifiable parameters. This changes the
symmetry and needs to be accounted for.  The concepts of this paper,
however, are compatible with this.  For example the value of the
determinant of a design is equivariant:
\begin{prop}\label{p:symmetric}
  Let $\sigma \in S_{m}$ and let $\xi$ be any design.  Then $\xi$ is
  locally $D$-optimal for the parameters
  $\beta^T=(\beta_1,\dots,\beta_{m-1})$ if and only if $\xi^{\sigma}$
  is locally $D$-optimal for $Q_\sigma^{-T} \beta $, where
  $\sigma \mapsto Q_\sigma $ is a group homomorphism from $S_m$ to the
  group of invertible $(m-1) \times (m-1)$-matrices satisfying
  $ f(\sigma(i),\sigma(j))=Q_\sigma f(i,j) $ for all $\sigma\in S_m$.
\end{prop}

\begin{proof}
	By \cite[Section~2]{Radloff2016}, the design $\xi^{\sigma}$ is
	locally optimal for the parameter $Q_\sigma^{-T} \beta $ if and only
	if there exist matrices $Q_\sigma$ as in the statement.  As
	transpositions generate all permutations, it suffices to show the
	existence of such a $Q_\sigma$ for all transpositions.  For
	transpositions of $i<m$ and $j<m$, let $Q_\sigma$ be the usual
	permutation matrix.  For a transposition $(i m)$, let $Q_\sigma$
	equal an identity matrix, with the $i$-th row replaced by the row
	$(-1\ldots,-1)$.  Then, for an arbitrary permutation $\sigma$, it
	holds that $f(\sigma(i),\sigma(j))=Q_\sigma f(i,j)$.
\end{proof}

\section{Saturated designs and graph-representation}
\label{s:graphrepresentation}
An experimental design is \emph{saturated} if its support has
cardinality equal to the number of free parameters of the model.  In
our case of $ D $-optimality, if a design has support size strictly
smaller than $m-1$, the determinant of the information matrix vanishes
and optimality is impossible.  A useful result about saturated designs
is that their weights are completely rigid: they are all equal
\cite[Lemma 5.1.3]{silvey1980optimal} and thus only the different
supports are considered.  We first study which saturated designs can
be $D$-optimal.  The following simple fact is reminiscent of the
\emph{connectedness} of block designs with block length two in
\cite[p.2]{shah1989}.
\begin{lemma}\label{l:trees}
  For any locally $D$-optimal saturated design $\xi$ of the
  Bradley--Terry paired comparison model, the graph representation of
  the support is a tree.
\end{lemma}

\begin{proof}
	A saturated design consists of $m-1$ equally weighted comparisons.  If
	there is a cycle $i_1,\dots,i_k$ in the graph representation of the
	design, then there is at least one alternative that does not appear in
	the design and therefore is represented by a disconnected vertex in the
	graph representation.  Now, the $ (m-1)\times(m-1)$-information matrix
	of a saturated design is a sum of $m-1$ rank one matrices of the form
	$\lambda_{ij}f(i,j)f(i,j)^{T}$. For $ 1\le i<j\le m-1 $, these rank
	one matrices only have entries in the $ i $-th and $ j $-th rows and
	columns. For $ j=m $, there is only one entry $ \lambda_{im} $ in the
	intersection of the $ i $-th row and $ i $-th column. Thereby, if a
	saturated design contains a cycle and misses one alternative that is not $ m $, the
	information matrix has no non-zero entries in either the corresponding
	row or the corresponding column. If alternative $ m $ is missed, it follows that every row sum of the information matrix is zero, as all rank one matrices are of the form $ \lambda_{ij} (e_i-e_j)(e_i-e_j)^{T} $. 
	Therefore the determinant of the
	information matrix is zero, and the design can never be optimal.
\end{proof}

Based on this fact we can determine the saturated optimal designs for
the Bradley--Terry model.

\begin{thm}
\label{t:path}
In the Bradley--Terry paired comparison model with $m$ alternatives,
if a design is saturated and locally D-optimal, then its graph
representation is a path on~$[m]:=\{1,2,\ldots,m\}$.
\end{thm}

\begin{proof}
  Let $ \xi $ be a saturated, locally $D$-optimal design for the
  Bradley--Terry model with $m$ alternatives.  The graph
  representation of $\xi$ is a tree by Lemma~\ref{l:trees}.  Applying
  a suitable permutation of $[m]$ and Proposition~\ref{p:symmetric},
  we assume that $\xi$ has exactly one comparison that contains~$m$,
  that is, that $m$ is a leaf.  Let $F$ be the (square) matrix of the
  transposed regression vectors of the design points
	\[F=
	\begin{pmatrix}
	f(i_1,j_1)^{T} \\
	f(i_2,j_2)^{T} \\
	\vdots \\
	f(i_{m-2},j_{m-2})^{T}\\
	f(i_{m-1},m)^{T}\\
	\end{pmatrix},
	\]
	and define
	$Q = \text{diag}(\lambda_{i_1,j_1},\dots,\lambda_{i_{m-1},m})$ as a
	diagonal matrix of intensities and correspondingly
	$W=\text{diag}(w_{i_1,j_1},\dots,w_{i_{m-1},m})$ for the weights of
	the design points.  Then, the information matrix is
	$
	M(\xi,\beta)=F^{T}WQF,
	$
	and inserting this into \eqref{eq:directionalder}, we obtain the directional 
	derivatives for every $ 1\le i <j\le m-1 $ as
	\[
	\lambda_{ij}f(i,j)^{T}F^{-1}Q^{-1}W^{-1}F^{-T}f(i,j)-(m-1).
	\]
	If the design is $D$-optimal, this
	formula is non-positive for every $ 1\le i <j\le m-1 $.  Since all
	weights are equal to $\frac{1}{m-1}$ this is equivalent to
	\[
	\lambda_{ij}f(i,j)^{T}F^{-1}Q^{-1}F^{-T}f(i,j)\le 1.
	\]
	
	The proof is by downward induction.  To this end, we remove one
	alternative and its associated design point and show that the reduced
	design $\bar{\xi}$ is optimal on the reduced design space.  Without
	loss of generality we can assume that the optimal design has only one
	comparison $ (1,v) $ in which alternative $1$ is involved. We can also assume
	that $v=2$ using the $S_m$ symmetry and Proposition~\ref{p:symmetric}.
	We remove alternative~$1$.  Consider the Bradley--Terry model on the
	alternatives $\{2,\dots,m\}$.  Its information matrix is a
	product $\bar{F}\bar{W}\bar{Q}\bar{F}^{T}$, where $ \bar{W} $ and
	$ \bar{F} $ are the lower-right $(m-2)\times (m-2)$-submatrices of
	$ \frac{m-1}{m-2}W $ and $F$, respectively, and $\bar{Q}$ is the
	diagonal matrix of the reduced model's intensities
	$\bar{\lambda}_{ij}$.  Through our assumptions,
	\[
	F=\begin{pmatrix}
	1&-1&0&\dots&0 \\
	0&&&&\\
	\vdots&&\bar{F}&& \\   
	0&&&&\\ 
	\end{pmatrix}.
	\]
	We show the implication
	\begin{align*}
	\lambda_{ij}f(i,j)^{T}F^{-1}Q^{-1}F^{-T}f(i,j)&\le 1 \quad \text{for all $ 2\le i<j\le m $}\\
	\Rightarrow \bar{\lambda}_{ij}\bar{f}(i,j)^{T}\bar{F}^{-1}\bar{Q}^{-1}\bar{F}^{-T}\bar{f}(i,j)&\le 1 \quad \text{for all $ 2\le i<j\le m $}.
	\end{align*}
	This implies that the design $\bar{\xi}$ with equal weights
	$\frac{1}{m-2}$ on $E\setminus\{1,2\}$ is optimal for the reduced
	model.  Since $ \bar{\lambda}_{ij}=\lambda_{ij} $, we only have to
	show
	\begin{align}
	\bar{f}(i,j)^{T}\bar{F}^{-1}\bar{Q}^{-1}\bar{F}^{-T}\bar{f}(i,j)&\le f(i,j)^{T}F^{-1}Q^{-1}F^{-T}f(i,j)\label{eq:reduction}
	\end{align}
	for all $ 2\le i<j\le m $.  Now let
	\[ F^{-1}=\begin{pmatrix}
	a_{11}& a_{12}^{T}\\
	a_{21}& A_1\\
      \end{pmatrix}\]
    for some $(m-2)\times(m-2) $-matrix~$A_1$. This leads to
    $ \bar{F}^{-1}=A_1-\frac{1}{a_{11}}a_{21}a_{12}^{T} $.  It can be
    checked that $a_{21}=0$ and thus
	\[F^{-1}=\left(\begin{array}{ccccccc}
	1&a_{12}^{T} \\
	0&A_{1}\\
	\end{array}\right). \]
	This means, that $ \bar{F}^{-1}=A_1 $.
	Now, as $ f(i,j)^T=(0, \bar{f}(i,j)^{T}) $,
	\begin{align*}
	f(i,j)^{T}&F^{-1}Q^{-1}F^{-T}f(i,j)\\ 
	& = \begin{pmatrix}
	0&\bar{f}(i,j)^{T}
	\end{pmatrix}
	\begin{pmatrix}
	1& a_{12}^{T}\\
	0& A_1\\
	\end{pmatrix}\begin{pmatrix}
	\frac{1}{\lambda_{12}}&\\
	&\bar{Q}^{-1}\\
	\end{pmatrix}\begin{pmatrix}
	1& 0\\
	a_{12}& A_1^{T}\\
	\end{pmatrix}\begin{pmatrix}
	0\\
	\bar{f}(i,j)
	\end{pmatrix}\\
	&=\begin{pmatrix}
	0  & \bar{f}(i,j)^{T} A_1
	\end{pmatrix}\begin{pmatrix}
	\frac{1}{\lambda_{12}}&\\
	&\bar{Q}^{-1}\\
	\end{pmatrix}\begin{pmatrix}
	0  \\
	A_1^{T} \bar{f}(i,j)\\
	\end{pmatrix}\\
	&=\bar{f}(i,j)^{T}A_1\bar{Q}^{-1} A_1^{T}\bar{f}(i,j)\\
	&=\bar{f}(i,j)^{T}\bar{F}^{-1}\bar{Q}^{-1} \bar{F}^{-T}\bar{f}(i,j).\\
	\end{align*}
	In fact, \eqref{eq:reduction} is realized as an equality and the
	reduced saturated design is optimal.
	Now if $\xi$ was not a path, iterating this procedure
        eventually leads to an optimal saturated design for the
        Bradley--Terry model on four alternatives that is not a path.
        Such a design does not exist by the explicit computations in
        Section~\ref{s:fouralternatives}.  Hence, the graph
        representation of a saturated, locally $D$-optimal design is a
        path.
\end{proof}

\section{Optimality Regions of saturated designs}
\label{s:optregion}

We now describe the sets of parameters for which a saturated design
from Theorem~\ref{t:path} is optimal.  We call such a set the
\emph{region of optimality} of the design.  Knowing these regions
simplifies the experimental design problem since it can be combined
with prior knowledge about the parameters (e.g.~from a screening
experiment).  Also, knowing if the regions are big or small yields
information about the robustness of designs.

Exploiting the symmetry in Proposition~\ref{p:symmetric}, it suffices
to study a single design representing all saturated designs.  This is
the path $(1,2),\allowbreak(2,3),\dots,\allowbreak (m-1,m)$.

\begin{lemma} \label{l:saturatedpath} The optimality region of the
  design $ (12,23,34,\ldots,(m-1)m) $ is defined by the inequalities
\[
  g(i,j)=\lambda_{ij}\sum_{k=i}^{j-1}\frac{1}{\lambda_{k(k+1)}} \le 1,
  \qquad 1\le i<m,\, i<j\le m.
\]
Furthermore, this region is not empty.
\end{lemma}

\begin{proof}
	We apply Theorem \ref{t:silvey} to find
	the optimality regions of the design $ (12,23,34,\ldots,(m-1)m)$. Therefore, one
	has to analyze the directional derivatives
	$
	f(i,j)^{T}F^{-1}Q^{-1}F^{-T}f(i,j)-(m-1),
	$
	where $f(i,j)$ are the regression vectors, $Q$ is a diagonal matrix of
	the design intensities
	$\lambda_{12}$,$\lambda_{23}$, \ldots, $\lambda_{(m-1)m}$ and $ F $ is the
	matrix of the transposed regression vectors.  So,
	\begin{align*}
	F &=\begin{pmatrix}
	1&&-1&& &&\\
	&&\ddots&&\ddots&&\\
	&& &&1&&-1\\
	&& && &&1\\
	\end{pmatrix} \qquad \text{and} \qquad
	F^{-1}=\begin{pmatrix}
	1&&\ldots&&1\\
	&&\ddots&&\vdots\\
	&& &&1\\
	\end{pmatrix}.
	\end{align*}
	For $ i<j<m $, we have $ f(i,j)=e_{i}-e_j $. This leads to
	\begin{align*}
	f(i,j)^{T}F^{-1}=(\mathbbm{1}_{\{i= 1\}},\mathbbm{1}_{\{i\le 2< j\}},\mathbbm{1}_{\{i\le 3< j\}},\ldots,\mathbbm{1}_{\{i\le m-2< j\}},0).
	\end{align*}
	For $ i<j=m $, we have $ f(i,m)=e_i $. So
	\begin{align*}
	f(i,m)^{T}F^{-1}=(\mathbbm{1}_{\{i= 1\}},\mathbbm{1}_{\{i\le 2\}},\mathbbm{1}_{\{i\le 3\}},\ldots,\mathbbm{1}_{\{i\le m-2\}},1).
	\end{align*}
	This means that the directional derivative in the direction $ (i,j) $ for $ j<m $ is
	\begin{align*}
	\lambda_{ij}(m-1)&(\mathbbm{1}_{\{i\le 1\}},\mathbbm{1}_{\{i\le 2<
		j\}},\ldots,\mathbbm{1}_{\{i\le m-2< j\}},0)
	 \begin{pmatrix}
	\frac{1}{\lambda_{12}}&&&&\\
	&\frac{1}{\lambda_{23}}&&&\\
	&&\ddots&&\\
	&&&\ddots&\\
	&&&&\frac{1}{\lambda_{(m-1)m}}\\
	\end{pmatrix}
	\begin{pmatrix}
	\mathbbm{1}_{\{i\le 1\}}\\
	\mathbbm{1}_{\{i\le 2< j\}}\\
	\vdots\\
	\mathbbm{1}_{\{i\le m-2< j\}}\\
	0\\
	\end{pmatrix}\\
	&=\lambda_{ij}(m-1)\sum_{k=1}^{m-2}\frac{\mathbbm{1}_{\{i\le k< j\}}}{\lambda_{k(k+1)}}\\
	&=\lambda_{ij}(m-1)\sum_{k=i}^{j-1}\frac{1}{\lambda_{k(k+1)}}\\
	\end{align*}
	and for $ j=m $ is
	\begin{align*}
	\lambda_{im}(m-1)&(\mathbbm{1}_{\{i\le 1\}},\mathbbm{1}_{\{i\le
		2\}},\ldots,\mathbbm{1}_{\{i\le m-2\}},1)	 \begin{pmatrix}
	\frac{1}{\lambda_{12}}&&&&\\
	&\frac{1}{\lambda_{23}}&&&\\
	&&\ddots&&\\
	&&&\ddots&\\
	&&&&\frac{1}{\lambda_{(m-1)m}}\\
	\end{pmatrix}
	\begin{pmatrix}
	\mathbbm{1}_{\{i\le 1\}}\\
	\mathbbm{1}_{\{i\le 2\}}\\
	\vdots\\
	\mathbbm{1}_{\{i\le m-2\}}\\
	1\\
	\end{pmatrix}\\
	&=\lambda_{im}(m-1)\sum_{k=1}^{m-1}\frac{\mathbbm{1}_{\{i\le k\}}}{\lambda_{k(k+1)}}\\
	&=\lambda_{im}(m-1)\sum_{k=i}^{m-1}\frac{1}{\lambda_{k(k+1)}}.
	\end{align*}
	For $ j=i+1 $ the directional derivatives are $0$ by
	\eqref{eq:directionalder} and Corollary~\ref{c:silvey}.  Let
	\[g(i,j)=\lambda_{ij}\sum_{k=i}^{j-1}\frac{1}{\lambda_{k(k+1)}}.\]
        By Theorem~\ref{t:silvey} the optimality region of the design
        $ (12, 23, 34, \ldots, (m-1)m) $ is cut out by the
        inequalities $g(i,j)\le 1$ for $1\le i<j\le m$.

        To exhibit a point in the optimality region, let
        $ \beta_i= i \beta_1$ and thus $\pi_i=\pi_1^i$. This implies
	\[\lambda_{ij}=\frac{\pi_1^{j-i}}{(1+\pi_1^{j-i})^2},\]
	and therefore
	\begin{align*}
	g(i,j) =\frac{\pi_1^{j-i}}{(1+\pi_1^{j-i})^2}\sum_{k=i}^{j-1}
	\frac{(1+\pi_1)^2}{\pi_1} 
	=\frac{(j-i)\pi_1^{j-i-1}(1+\pi_1)^2}{(1+\pi_1^{j-i})^2},
	\end{align*}
	which is at most $ 1 $ for all $ 1\le i<j\le m $ if just
	$\pi_1$ is sufficiently large.
\end{proof}

\begin{thm}\label{t:regions}
The optimality regions of all saturated designs corresponding to
paths, i.e.~of all optimal saturated designs, are in the $S_m$-orbit
of the saturated design for $(12,23,34,\ldots,(m-1)m)$.  The
optimality regions are defined by the inequalities
\[
\{g(\sigma(i),\sigma(j))\le 1 \,:\, 1\le i<m,
\,i<j\le m \}.
\]
where $\sigma\in S_m$ is a permutation turning $(12,23,34,\ldots,(m-1)m)$
into the given path.
\end{thm}

\begin{proof}
	Theorem~\ref{t:path} shows that the saturated optimal designs
	correspond to paths. By Proposition~\ref{p:symmetric}, we can choose
	any representative for the orbit of path designs.  We choose
	$(12,23,34,\ldots,(m-1)m)$ and plug in the results from Lemma \ref{l:saturatedpath}.
\end{proof}

\section{Explicit solutions for four alternatives}\label{s:fouralternatives}
This section studies the optimal designs for the Bradley--Terry paired
comparison model with four alternatives, as it arises for example
in~\cite{gabrielsen2000paired}.  We first deal with the case of
saturated designs, i.e.\ optimal designs whose supports consist of
only 3 design points.  The unsaturated case with 4, 5 or 6 support
points follows in Section~\ref{s:unsaturated4}.

The Bradley--Terry paired comparison model with $4$ alternatives has
$3$ identifiable parameters $\beta_1,\beta_2,\beta_3$.  As above we
use $\beta_i:=\log(\pi_i)$ and $\beta_4=0$.  Our goal is to cover all
of $\RR^3$ with regions of optimality of specific explicit designs.
The regression vectors for four alternatives are
\begin{align*}
f(1,2)&= (1, -1, 0)^{T},&		f(1,3) &= (1, 0, -1)^{T},&    f(1,4)&= (1, 0, 0)^{T},\\		f(2,3) &= (0, 1, -1)^{T}&       f(2,4)&= (0, 1, 0)^{T},&	f(3,4) &= (0, 0, 1)^{T}.
\end{align*}

\subsection{Saturated Designs }
For saturated designs with non-singular information matrices, the
optimality criterion in Theorem~\ref{t:silvey} yields a system of
inequalities in the intensities $\lambda_{ij}$.  We find these first.
According to \cite[Lemma 5.1.3]{silvey1980optimal}, a saturated design
has three positive weights whose values are all~$\tfrac{1}{3}$, the
remaining weights being zero.  There are $ \binom{6}{3}=20$ possible
saturated designs.  Exactly $16$ of them have a non-singular
information matrix.  Among the $16$, only $12$ have a non-empty region
of optimality.  We find that they are in bijection with the paths on 4
vertices.  The following theorem is the base case to which the proof
of Theorem~\ref{t:path} reduces.

\begin{thm}\label{thm:paths4} For the Bradley--Terry model with four
alternatives there are $20$ saturated designs.  Among those
\begin{itemize}
\item $8$ have an empty region of optimality.
\item $12$ have optimal experimental designs.
\end{itemize}
The $12$ designs with non-empty region of optimality correspond to the
$12$ labelings of the path $P_4$.  The region of optimality of the
path $(1,2),(2,3),(3,4)$ is constrained by
\begin{align}
\lambda_{13}(\lambda_{12}+\lambda_{23})-\lambda_{12}\lambda_{23}&\le 0,\nonumber\\
\lambda_{24}(\lambda_{23}+\lambda_{34})-\lambda_{23}\lambda_{34}&\le 0, \label{ineq1} \\
\lambda_{14}(\lambda_{12}\lambda_{23}+\lambda_{12}\lambda_{34}+\lambda_{23}\lambda_{34})-\lambda_{12}\lambda_{23}\lambda_{34}&\le 0. \nonumber
\end{align}
The regions of optimality for other paths arise from this by
relabeling.
\end{thm}

\tikzset{%
    line width=2mm
}
\begin{figure}[ht]
\centering
\subfloat{
\begin{tikzpicture}[line width=.6mm]
 \node[shape=circle,draw=black] (1) at (-1,1) {1};
  \node[shape=circle,draw=black] (2) at (1,1) {2};
   \node[shape=circle,draw=black] (3) at (-1,-1) {3};
    \node[shape=circle,draw=black] (4) at (1,-1) {4};
    \path (1) edge (2);
    \path (2) edge (3);
    \path (3) edge (4);
\end{tikzpicture}}
\qquad\qquad
\subfloat{
\begin{tikzpicture}[line width=.6mm]
 \node[shape=circle,draw=black] (1) at (-1,1) {1};
  \node[shape=circle,draw=black] (2) at (1,1) {2};
   \node[shape=circle,draw=black] (3) at (-1,-1) {3};
    \node[shape=circle,draw=black] (4) at (1,-1) {4};
    \path (1) edge (2);
    \path (1) edge (3);
    \path (1) edge (4);
\end{tikzpicture}}
\qquad\qquad
\subfloat{
\begin{tikzpicture}[line width=.6mm]
 \node[shape=circle,draw=black] (1) at (-1,1) {1};
  \node[shape=circle,draw=black] (2) at (1,1) {2};
   \node[shape=circle,draw=black] (3) at (-1,-1) {3};
    \node[shape=circle,draw=black] (4) at (1,-1) {4};
    \path (1) edge (2);
    \path (1) edge (4);
    \path (2) edge (4);
\end{tikzpicture}}

\caption{\label{fig:graphs4}
Graph representations of different $ 3 $-point designs.}
\end{figure}
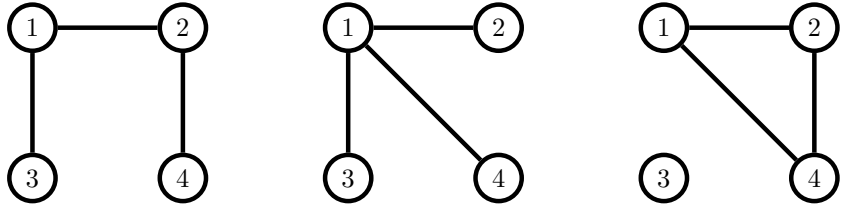

Since the $D$-optimality criterion is equivariant under the $S_4$ action
by Proposition~\ref{p:symmetric}, it suffices to study one labeling for
each unlabeled graph with three edges on four vertices.  The proof of
Theorem~\ref{thm:paths4} is split into a discussion of information
matrices for the three graphs in Figure~\ref{fig:graphs4}.

\subsubsection{Paths}\label{s:paths4}
Consider the path in Figure~\ref{fig:graphs4}.  Its edge set is
$\{(1,2),(2,3),(3,4)\}$.  A corresponding saturated design can only be
optimal if its weights are $ w_{12}=w_{23}=w_{34}=\frac{1}{3} $ and
$w_{13}=w_{14}=w_{24}=0$.  The information matrix of this design is
\begin{align*}
M&=\frac{1}{3}(\lambda_{12} f(1,2)f(1,2)^{T}+\lambda_{23} f(2,3)f(2,3)^{T}+\lambda_{34} f(3,4)f(3,4)^{T})\\
&=\frac{1}{3}\left(
\begin{array}{ccc}
 \lambda_{12} & -\lambda_{12} & 0 \\
 -\lambda_{12} & \lambda_{12}+ \lambda_{23} & -\lambda_{23} \\
 0 & -\lambda_{23} & \lambda_{23}+\lambda_{34} \\
\end{array}
\right) .
\end{align*}
We apply Theorem \ref{t:silvey}. The directional derivatives are
\[g_{ij}(\lambda) := \lambda_{ij}f(i,j)^T M^{-1} f(i,j)-3 .\]
The region of optimality is
\[
\{ \lambda \in \RR^{\mathcal{X}}_{> 0} : g_{ij}(\lambda) \le 0,\; 1\le i < j
\le 4 \}.
\]
This region is a semi-algebraic set, that is, defined constructively
by polynomial inequalities.  To see this we use \Mathematica.
Corollary~\ref{c:silvey} simplifies the description because it says
that for design points with positive weights the conditions become
equations, and those equations have no free variables, as the weights
in a saturated design are fixed.  Using \Mathematica{}'s \verb|Reduce|
functionality we derived~\eqref{ineq1}.

The inequalities in \eqref{ineq1} can be compared to
\cite[Theorem~2]{Grasshoff2008}.  The structure is similar, but for
four alternatives a cubic inequality appears.  For more alternatives
even higher degree inequality constraints appear according to
Theorem~\ref{t:regions}.  These conditions can be expressed in
$\beta$-coordinates.  The resulting regions of optimality are
displayed in Figure~\ref{fig:redblue} on the left.

\begin{figure}[ht]
\centering
\includegraphics[scale=0.6]{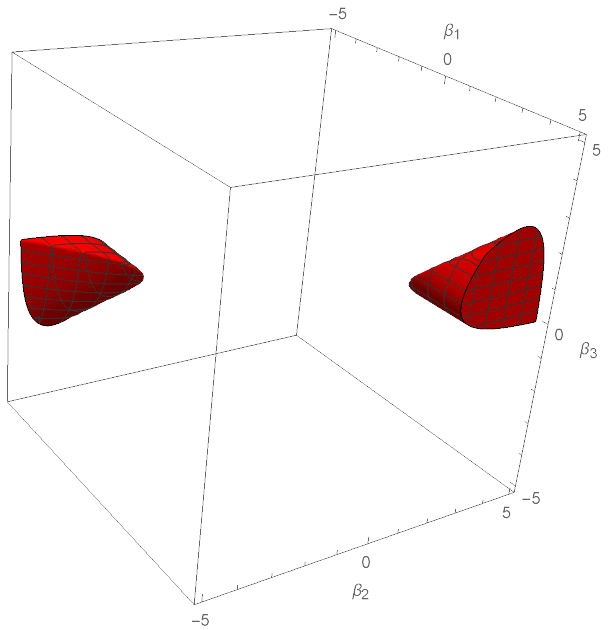}
\includegraphics[scale=0.5]{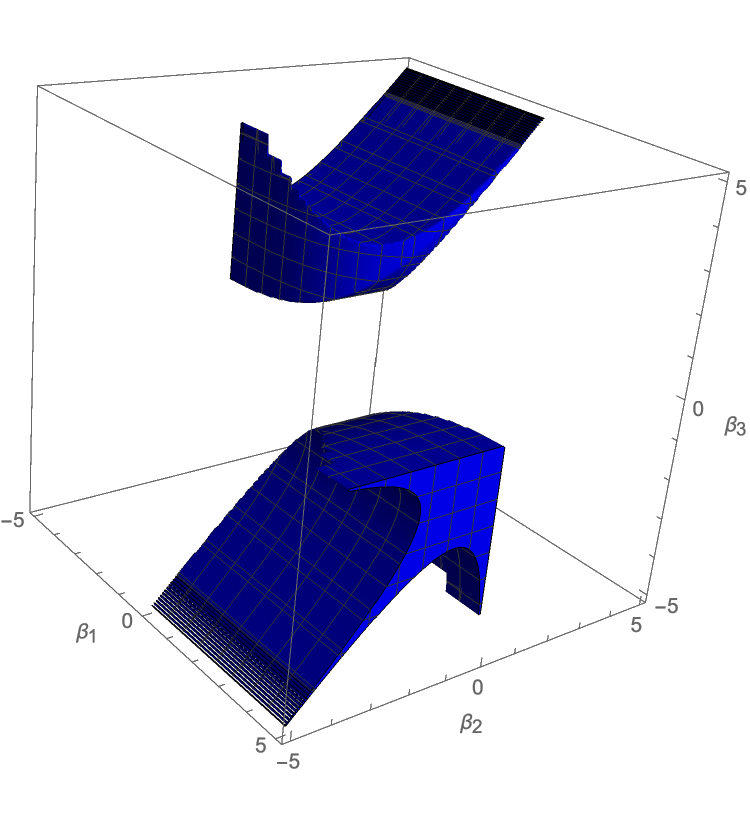}
\caption{\label{fig:redblue} Optimality regions for the saturated
design on $(12,23,34)$ on the left and $4$-point designs with
$w_{12}=w_{13}=0$ on the right.}
\end{figure}

\subsubsection{The claw graph \texorpdfstring{$K_{1,3}$}{K13}}\label{s:claws4}
We now show that the graph in the middle of Figure~\ref{fig:graphs4},
sometimes known as a \emph{claw}, leads to an empty region of
optimality.  After symmetry reduction it suffices to show that the
design $(12,13,14)$ cannot be $D$-optimal.  This design would be
optimal in the following region given by the three directional
derivatives corresponding to the non-edges $(23,24,34)$:
\[
\lambda _{23}\leq \frac{\lambda _{12} \lambda _{13}}{\lambda
_{12}+\lambda _{13}}\land \lambda _{24}\leq \frac{\lambda _{12}
\lambda _{14}}{\lambda _{12}+\lambda _{14}}\land \lambda _{34}\leq
\frac{\lambda _{13} \lambda _{14}}{\lambda _{13}+\lambda _{14}}.
\]
Plugging in the formulas for the $\lambda_{ij}$ in terms of the
$\pi_i$ this becomes
\begin{align*}
&(\pi_2+\pi_3) \left(\pi_1^2+\pi_2 \pi_3\right)\leq \pi_1 (\pi_2-\pi_3)^2,\\
& (\pi_2+1) \left(\pi_1^2+\pi_2\right)\leq \pi_1 (\pi_2-1)^2,\\
& (\pi_3+1) \left(\pi_1^2+\pi_3\right)\leq \pi_1 (\pi_3-1)^2.
\end{align*}
Using \Mathematica, we find that these conditions are incompatible
with $\pi_1>0,\pi_2>0,\pi_3>0$.  It would be interesting to find a
short certificate for the infeasibility of this system.  Such a
certificate always exists by the Positivstellensatz from real
algebraic geometry (see \cite{Bochnak}).  This means that if the
inequality system has no solution, then one can combine the
inequalities to produce an explicit contradiction.  There are
computational tools to search for such certificates, but our attempts
with SOStools~\cite{sostools} were not successful.

\subsubsection{Singular designs}\label{s:singular4}
Designs corresponding to the rightmost graph in
Figure~\ref{fig:graphs4} have singular information matrices and can
thereby not be $ D $-optimal.

\begin{proof}[Proof of Theorem~\ref{thm:paths4}]
Since there are 12 distinct labelings of the path on four vertices,
the theorem follows from the computations in 
Sections~\ref{s:paths4}--\ref{s:singular4}.
\end{proof}

\subsection{Unsaturated Designs}
\label{s:unsaturated4}
We now examine the designs whose support contains at least four pairs.
In this case the weights $w_{ij}$ of an optimal design are not
necessarily uniform.  Instead we find formulas that express the
weights in terms of the parameters.  These formulas might look
complicated, but they are very symmetric and can easily be handled by
computer algebra systems.
Our approach is again via Theorem~\ref{t:silvey}: optimality of a
design $\xi^*$ is equivalent to
\begin{align}\label{eq:directionalderivative}
  \lambda_{ij}f(i,j)^T M(\xi^*,\beta)^{-1} f(i,j)-3 \le 0,\qquad 1\leq
  i<j\le 4.
\end{align}
Furthermore, by Corollary~\ref{c:silvey}, there is equality for any
pair $i,j$ such that $w_{ij} > 0$ in $\xi^*$.  We distinguish cases
according to the size of the support.

\subsubsection{Full support}\label{s:fourpointfullsupport}
Full support means that all weights of a design are positive.
Then all inequalities \eqref{eq:directionalderivative} hold with
equality and we have a system of 6 equations in the variables
$ w_{ij},\lambda_{ij} $ for $ 1\le i <j\le 4 $.  We used \Mathematica
to solve the system and to express the weights $ w_{ij} $ as functions
of the intensities $\lambda_{ij}$:
\begin{multline*}
  w_{ij} =\frac{1}{A}(\lambda_{ik} \lambda _{il} \lambda _{jk} \lambda _{jl} (\lambda _{ij} \lambda _{ik} \lambda _{il} \lambda _{jk} \lambda _{jl}-\lambda _{ij} \lambda _{ik} \lambda _{il} \lambda _{jk} \lambda _{kl}-\lambda _{ij} \lambda _{ik} \lambda _{il} \lambda _{jl} \lambda _{kl}
  -\lambda _{ij} \lambda _{ik} \lambda _{jk} \lambda _{jl} \lambda _{kl}\\
 +\lambda _{ij} \lambda _{ik} \lambda _{jk} \lambda
   _{kl}^2-\lambda _{ij} \lambda _{ik} \lambda _{jl} \lambda
   _{kl}^2-\lambda _{ij} \lambda _{il} \lambda _{jk} \lambda _{jl}
   \lambda _{kl}-\lambda _{ij} \lambda _{il} \lambda _{jk} \lambda
   _{kl}^2
   +\lambda _{ij} \lambda _{il} \lambda _{jl} \lambda _{kl}^2+2
   \lambda _{ik} \lambda _{il} \lambda _{jk} \lambda _{jl} \lambda
   _{kl})),
\end{multline*}
where $ (i,j,k,l) $ is any permutation of $ (1,2,3,4) $. The term $ A
$ is the normalization that ensures $\sum_{i<j}w_{ij} = 1$.
\begin{multline*}
  A  = 3 (\lambda_{ij} \lambda_{ik}^2 \lambda_{il}^2 \lambda_{jk}^2
      \lambda_{jl}^2+\lambda_{ij} \lambda_{ik} \lambda_{il}^2
      \lambda_{jk} \lambda_{jl}^2 \lambda_{kl}^2-\lambda_{ij}
      \lambda_{ik} \lambda_{il}^2 \lambda_{jk}^2 \lambda_{jl}
      \lambda_{kl}^2
      -\lambda_{ij}^2 \lambda_{ik} \lambda_{il}^2
      \lambda_{jk} \lambda_{jl} \lambda_{kl}^2 \\
       - \lambda_{ij} \lambda_{ik} \lambda_{il}^2 \lambda_{jk}^2
	\lambda_{jl}^2 \lambda_{kl}
	-\lambda_{ij} \lambda_{ik}^2 \lambda_{il}^2 \lambda_{jk}
	\lambda_{jl}^2 \lambda_{kl}
	-\lambda_{ij} \lambda_{ik}^2
	\lambda_{il}^2 \lambda_{jk}^2 \lambda_{jl}
	\lambda_{kl}-\lambda_{ij}^2 \lambda_{ik} \lambda_{il}^2
	\lambda_{jk}^2 \lambda_{jl} \lambda_{kl} 
  + \lambda_{ij}^2 \lambda_{ik}^2 \lambda_{il}^2 \lambda_{jk}
   \lambda_{jl} \lambda_{kl}\\-\lambda_{ij} \lambda_{ik}^2 \lambda_{il}
   \lambda_{jk} \lambda_{jl}^2 \lambda_{kl}^2 -\lambda_{ij}^2
   \lambda_{ik} \lambda_{il} \lambda_{jk} \lambda_{jl}^2
   \lambda_{kl}^2+\lambda_{ij} \lambda_{ik}^2 \lambda_{il}
   \lambda_{jk}^2 \lambda_{jl} \lambda_{kl}^2
   -\lambda_{ij}^2
   \lambda_{ik} \lambda_{il} \lambda_{jk}^2 \lambda_{jl}
   \lambda_{kl}^2-\lambda_{ij}^2 \lambda_{ik}^2 \lambda_{il}
   \lambda_{jk} \lambda_{jl} \lambda_{kl}^2\\-\lambda_{ij}
   \lambda_{ik}^2 \lambda_{il} \lambda_{jk}^2 \lambda_{jl}^2
    \lambda_{kl}
    +\lambda_{ij}^2 \lambda_{ik} \lambda_{il} \lambda_{jk}^2
    \lambda_{jl}^2 \lambda_{kl} 
     -\lambda_{ij}^2 \lambda_{ik}^2 \lambda_{il} \lambda_{jk} \lambda_{jl}^2 \lambda_{kl}+\lambda_{ij}^2 \lambda_{ik} \lambda_{il}^2 \lambda_{jk}^2 \lambda_{kl}^2
     +\lambda_{ij}^2 \lambda_{ik}^2 \lambda_{il} \lambda_{jl}^2 \lambda_{kl}^2\\+\lambda_{ij}^2 \lambda_{ik}^2 \lambda_{jk} \lambda_{jl}^2 \lambda_{kl}^2
+\lambda_{ij}^2 \lambda_{il}^2 \lambda_{jk}^2 \lambda_{jl} \lambda_{kl}^2
+\lambda_{ik}^2 \lambda_{il}^2 \lambda_{jk}^2 \lambda_{jl}^2 \lambda_{kl}).
\end{multline*} 
This design is locally optimal for some $\beta$ when $ w_{ij}> 0 $ for
all $ 1\le i <j \le 4 $. Figure~\ref{fig:65pointRegion} shows the
optimality region of $6$-point-designs on the left.
\begin{figure}[ht]
\centering
\includegraphics[scale=0.5]{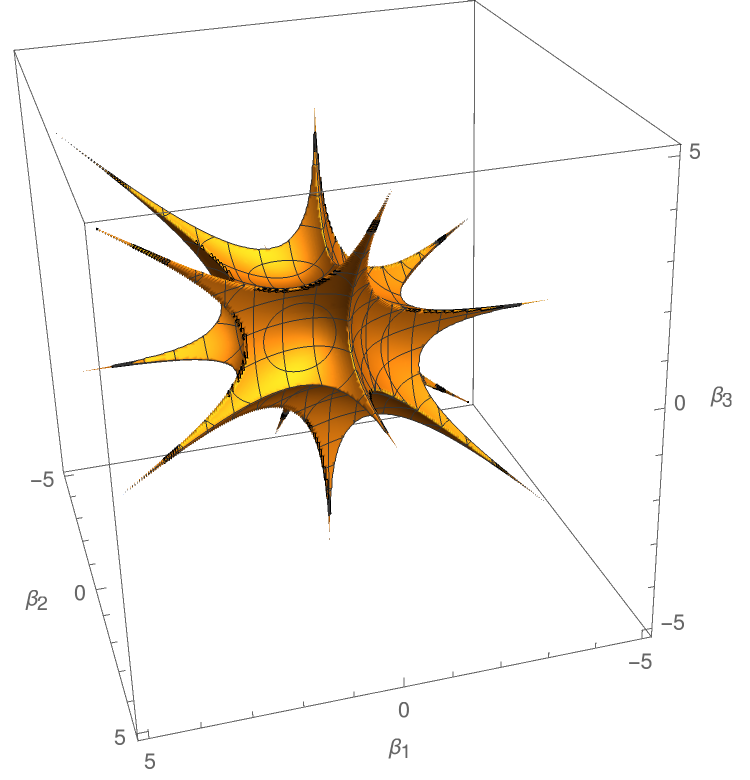}
\includegraphics[scale=0.5]{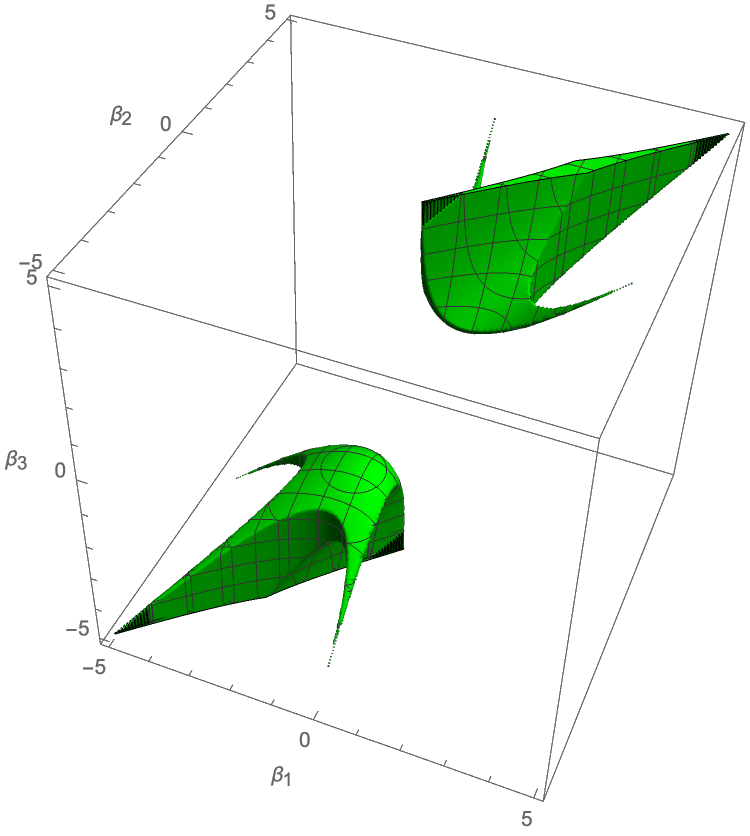}
\caption{\label{fig:65pointRegion}Optimality regions for
  $6$-point (left) and $5$-point designs with $ w_{34}=0 $ (right).}
\end{figure}

\begin{example}
  A simple example for a design with full support arises when
  $\beta_i=0$ for all $ 1\le i \le 4$.  Then
  $ \lambda_{ij}=\frac{1}{4} $ for all $ 1\le i < j \le 4 $ and
  therefore $ w_{ij}=\frac{1}{6}$, that is, assigning the same number
  of repetitions to each comparison, is optimal.
  Figure~\ref{fig:65pointRegion} and the continuity of the formulas
  for $w_{ij}$ illustrate that, whenever all $\beta_{i}$ are
  sufficiently small, an optimal design will assign almost equal
  number of repetitions to each pair~$(i,j)$.
\end{example}

 \begin{remark}
 When working with polynomial equations, Gr\"obner bases are a powerful
 tool.  The expressions of the $w_{ij}$ in terms of the $\lambda_{ij}$
 can also be found using elimination theory.  For example, the
 computer algebra system \Mtwo~\cite{M2} makes this easy.
 \end{remark}

\subsubsection{5-point designs}
We now discuss optimal designs where one weight is zero.  There is one
orbit under the action of $S_4$, that is, a permutation of the
alternatives transforms any given five-point design to the one that
does not use comparison~$(1,2)$.  Therefore we discuss the design with
$w_{12}=0$ and the remaining weights positive.  Then the optimality
conditions become
\begin{align*}
w_{13}&=\frac{2 \lambda _{14} \lambda _{34} \left(\lambda _{14} \lambda _{34}-\lambda _{13} \left(\lambda _{14}+\lambda _{34}\right)\right)}{3 \left(\lambda _{13}^2 \left(\lambda _{14}-\lambda _{34}\right){}^2-2 \lambda _{13} \lambda _{14} \lambda _{34} \left(\lambda _{14}+\lambda _{34}\right)+\lambda _{14}^2 \lambda _{34}^2\right)},\\
w_{14}&=\frac{2 \lambda _{13} \lambda _{34} \left(\lambda _{13} \left(\lambda _{34}-\lambda _{14}\right)-\lambda _{14} \lambda _{34}\right)}{3 \left(\lambda _{13}^2 \left(\lambda _{14}-\lambda _{34}\right){}^2-2 \lambda _{13} \lambda _{14} \lambda _{34} \left(\lambda _{14}+\lambda _{34}\right)+\lambda _{14}^2 \lambda _{34}^2\right)},\\
w_{23}&=\frac{2 \lambda _{24} \lambda _{34} \left(\lambda _{24} \lambda _{34}-\lambda _{23} \left(\lambda _{24}+\lambda _{34}\right)\right)}{3 \left(\lambda _{23}^2 \left(\lambda _{24}-\lambda _{34}\right){}^2-2 \lambda _{23} \lambda _{24} \lambda _{34} \left(\lambda _{24}+\lambda _{34}\right)+\lambda _{24}^2 \lambda _{34}^2\right)},\\
w_{24}&=\frac{2 \lambda _{23} \lambda _{34} \left(\lambda _{23} \left(\lambda _{34}-\lambda _{24}\right)-\lambda _{24} \lambda _{34}\right)}{3 \left(\lambda _{23}^2 \left(\lambda _{24}-\lambda _{34}\right){}^2-2 \lambda _{23} \lambda _{24} \lambda _{34} \left(\lambda _{24}+\lambda _{34}\right)+\lambda _{24}^2 \lambda _{34}^2\right)},\\
\end{align*}
and
\begin{multline*}
w_{34}=
\frac{1}{B}\Big(3 \lambda _{13}^2 \lambda _{14}^2 \lambda _{23}^2 \lambda _{24}^2-4 \lambda _{13} \lambda _{14} \lambda _{23} \lambda _{24} \lambda _{34}^4-2 \lambda _{13} \lambda _{14} \lambda _{23}^2 \lambda _{24}^2 \lambda _{34}^2
+4 \lambda _{13} \lambda _{14}^2 \lambda _{23} \lambda _{24}^2 \lambda _{34}^2\\
+4 \lambda _{13}^2 \lambda _{14} \lambda _{23} \lambda _{24}^2 \lambda
_{34}^2+4 \lambda _{13} \lambda _{14}^2 \lambda _{23}^2 \lambda _{24}
\lambda _{34}^2
+4 \lambda _{13}^2 \lambda _{14} \lambda _{23}^2
\lambda _{24} \lambda _{34}^2-2 \lambda _{13}^2 \lambda _{14}^2
\lambda _{23} \lambda _{24} \lambda _{34}^2
-4 \lambda _{13} \lambda _{14}^2 \lambda _{23}^2 \lambda _{24}^2 \lambda _{34}\\
-4 \lambda _{13}^2 \lambda _{14} \lambda _{23}^2 \lambda _{24}^2
\lambda _{34}-4 \lambda _{13}^2 \lambda _{14}^2 \lambda _{23} \lambda
_{24}^2 \lambda _{34}-4 \lambda _{13}^2 \lambda _{14}^2 \lambda
_{23}^2 \lambda _{24} \lambda _{34}
+2 \lambda _{13} \lambda _{14} \lambda _{23}^2 \lambda _{34}^4+\lambda
_{13}^2 \lambda _{14}^2 \lambda _{23}^2 \lambda _{34}^2 \\
+2 \lambda _{13} \lambda _{14} \lambda _{24}^2 \lambda _{34}^4+\lambda _{13}^2 \lambda _{14}^2 \lambda _{24}^2 \lambda _{34}^2
+2 \lambda _{13}^2 \lambda _{23} \lambda _{24} \lambda _{34}^4+\lambda
_{13}^2 \lambda _{23}^2 \lambda _{24}^2 \lambda _{34}^2-\lambda
_{13}^2 \lambda _{23}^2 \lambda _{34}^4-\lambda _{13}^2 \lambda
_{24}^2 \lambda _{34}^4\\
+2 \lambda _{14}^2 \lambda _{23} \lambda _{24} \lambda _{34}^4
+\lambda _{14}^2 \lambda _{23}^2 \lambda _{24}^2 \lambda _{34}^2-\lambda _{14}^2 \lambda _{23}^2 \lambda _{34}^4-\lambda _{14}^2 \lambda _{24}^2 \lambda _{34}^4\Big),\quad\quad\quad
\end{multline*}
with
\begin{multline*}
B=3 \left(\lambda _{13}^2 \lambda _{14}^2-2 \lambda _{13}^2 \lambda _{14} \lambda _{34}-2 \lambda _{13} \lambda _{14} \lambda _{34}^2-2 \lambda _{13} \lambda _{14}^2 \lambda _{34}+\lambda _{13}^2 \lambda _{34}^2+\lambda _{14}^2 \lambda _{34}^2\right) \\
\qquad \cdot \left(\lambda _{23}^2 \lambda _{24}^2-2 \lambda _{23}^2 \lambda _{24} \lambda _{34}-2 \lambda _{23} \lambda _{24} \lambda _{34}^2-2 \lambda _{23} \lambda _{24}^2 \lambda _{34}+\lambda _{23}^2 \lambda _{34}^2+\lambda _{24}^2 \lambda _{34}^2\right).
\end{multline*}
These designs are optimal if the directional derivative in
$(1,2)$-direction is smaller than or equal to zero, which is
equivalent to
\begin{multline*}
 \lambda _{12} (\lambda _{13} (\lambda _{14} (\lambda _{23} (\lambda _{24}-\lambda _{34})-\lambda _{24} \lambda _{34})+\lambda _{34} (\lambda _{23} (\lambda _{34}-\lambda _{24})-\lambda _{24} \lambda _{34}))
 -\lambda _{14} \lambda _{34} (\lambda _{23} (\lambda _{24}+\lambda _{34})-\lambda _{24} \lambda _{34}))\\
\ge -2 \lambda _{13} \lambda _{14} \lambda _{23} \lambda _{24} \lambda _{34}.
\end{multline*}
This inequality together with the formulas for the weights and the
condition, that all the weights except $ w_{12} $ are positive, gives
the design region.  This region is non-empty.  A plot 
in $\beta$-coordinates is on the right in Figure~\ref{fig:65pointRegion}.

\subsubsection{4-point designs}

We now discuss designs whose support contain exactly four
points. There are $ \binom{6}{4} = 15 $ possibilities for such designs
which each have two zero weights, $ w_{ij}=w_{kl}=0 $.  The four-point
designs form two orbits under the action of~$S_4$, distinguished by
whether the two non-edges in the graph representation share a vertex
or not, that is, whether $|\{i,j,k,l\}| = 4$, that is, $i,j,k,l$ are
all distinct, or $|\{i,j,k,l\}| = 3$, that is, exactly two are equal.
In the first case, there are three different design classes.  We
believe that these designs cannot be $D$-optimal, as the condition
$ w_{ij}=w_{kl}=0 $ with $|\{i,j,k,l\}| = 4$ implies that a third
weight is zero, which would lead to a saturated design.  A proof of
this statement eludes us so far.  Using \Mathematica, it follows from
the equivalence theorem that such a design satisfies
\begin{align}\label{eq:fourpointugly}
\lambda_{ik}(w_{ik}^2-\frac{w_{ik}}{3})=\lambda_{il}(w_{il}^2-\frac{w_{il}}{3})=\lambda_{jk}(w_{jk}^2-\frac{w_{jk}}{3})=\lambda_{jl}(w_{jl}^2-\frac{w_{jl}}{3}),
\end{align}
with $0 < w_{ik}, w_{il}, w_{jk}, w_{jl} < \frac{1}{3}$ and additionally
the inequalities
\begin{align}
\frac{\lambda_{ij}(3(w_{il}+w_{jl})-2)(3(w_{il}+w_{jl})-1)}{\lambda_{jl}w_{jl}(3w_{jl}-1)}&\le 3,\label{ineq:fourpointugly1}\\
\frac{\lambda_{kl}(3(w_{jk}+w_{jl})-2)(3(w_{jk}+w_{jl})-1)}{\lambda_{jl}w_{jl}(3w_{jl}-1)}&\le 3.\label{ineq:fourpointugly2}
\end{align}
Among the solutions of \eqref{eq:fourpointugly} there are the
saturated designs.  If one of the weights equals $1/3$, then
\eqref{eq:fourpointugly} implies that another weight is zero, i.e.\
the design is saturated.  Since the saturated cases have been dealt
with in Theorem~\ref{thm:paths4}, we only look for solutions whose weights all
lie in the open interval $(0,1/3)$.  There are solutions
of~\eqref{eq:fourpointugly} that satisfy this, for example, if the
weights and corresponding intensities are equal.  In all the cases we
examined, the inequalities \eqref{ineq:fourpointugly1} and
\eqref{ineq:fourpointugly2} are not satisfied.
\begin{prob}
\label{prob:infeasible}
Show that independent of the $\lambda_{ij}$, a simultaneous solution
of \eqref{eq:fourpointugly}, \eqref{ineq:fourpointugly1}, and
\eqref{ineq:fourpointugly2} is a saturated design.
\end{prob}

Finally we analyze the orbit of four-point-designs with
$w_{ij}=w_{kl}=0$ with $|\{i,j,k,l\}| = 3$.  Consider the
representative with $w_{12} = w_{13}=0$.  Then,
\begin{align*}
w_{14}&=\frac{1}{3}\\
w_{23}&=\frac{2 \lambda _{24} \lambda _{34} \left(-\lambda _{23} \lambda _{24}-\lambda _{23} \lambda _{34}+\lambda _{24} \lambda _{34}\right)}{3 \left(\lambda _{23}^2 \lambda _{24}^2-2 \lambda _{23}^2 \lambda _{24} \lambda _{34}-2 \lambda _{23} \lambda _{24} \lambda _{34}^2-2 \lambda _{23} \lambda _{24}^2 \lambda _{34}+\lambda _{23}^2 \lambda _{34}^2+\lambda _{24}^2 \lambda _{34}^2\right)}\\
w_{24}&=\frac{2 \lambda _{23} \lambda _{34} \left(-\lambda _{23} \lambda _{24}+\lambda _{23} \lambda _{34}-\lambda _{24} \lambda _{34}\right)}{3 \left(\lambda _{23}^2 \lambda _{24}^2-2 \lambda _{23}^2 \lambda _{24} \lambda _{34}-2 \lambda _{23} \lambda _{24} \lambda _{34}^2-2 \lambda _{23} \lambda _{24}^2 \lambda _{34}+\lambda _{23}^2 \lambda _{34}^2+\lambda _{24}^2 \lambda _{34}^2\right)}\\
w_{34}&=\frac{2 \lambda _{23} \lambda _{24} \left(\lambda _{24} \lambda _{24}-\lambda _{23} \lambda _{34}-\lambda _{24} \lambda _{34}\right)}{3 \left(\lambda _{23}^2 \lambda _{24}^2-2 \lambda _{23}^2 \lambda _{24} \lambda _{34}-2 \lambda _{23} \lambda _{24} \lambda _{34}^2-2 \lambda _{23} \lambda _{24}^2 \lambda _{34}+\lambda _{23}^2 \lambda _{34}^2+\lambda _{24}^2 \lambda _{34}^2\right)}
\end{align*}
This design is optimal if the directional derivatives along $ (1,2) $
and $ (1,3) $ are smaller than 3, so if
\begin{align*}
\frac{3 \lambda_{12} (\lambda_{14}+\lambda_{24})}{\lambda_{14} \lambda_{24}}\le 3~
\land ~\frac{3 \lambda_{13} (\lambda_{14}+\lambda_{34})}{\lambda_{14} \lambda_{34}}\le 3.
\end{align*}
This optimality region for this $4$-point design is visualized in
Figure~\ref{fig:redblue} on the right.  For each point in the
optimality region, the specific weights are computed by the equations
above.

Having discussed all cases, it suffices to apply the symmetry to each
of these regions and then $\RR^3$ can be pieced together.
Figure~\ref{fig:regionpuzzle} gives an idea of this puzzle.  
Because of continuity, the boundaries between any two regions always belong to the region with fewer design points.  Therefore, the yellow amoeba is open, the red regions for saturated
designs are closed (by the non-strict inequalities in
Theorem~\ref{thm:paths4}), and all other regions have both open and
closed boundaries.
\begin{figure}[ht]
\centering
\includegraphics[scale=0.8]{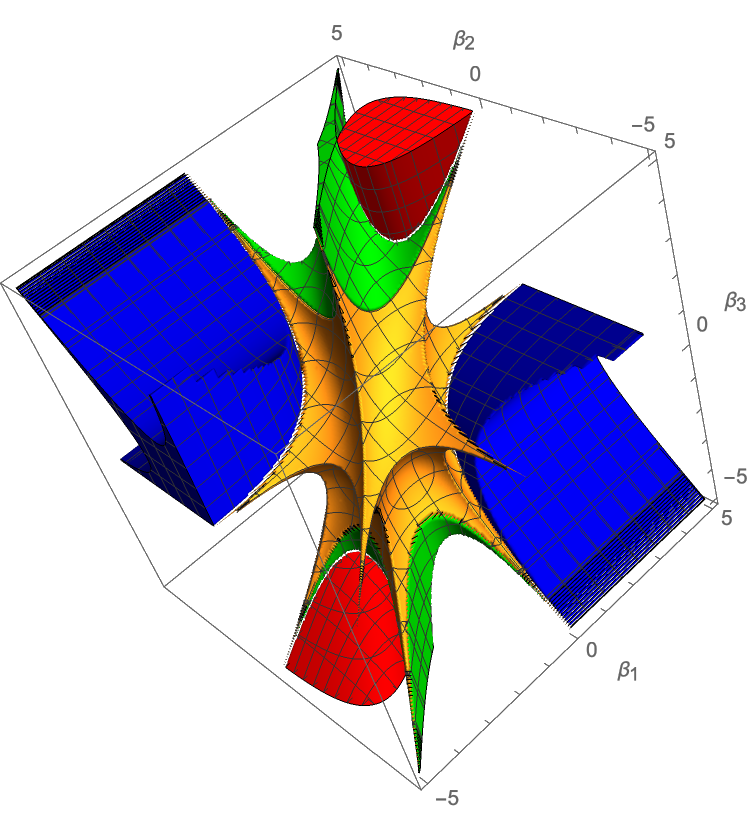}
\caption{\label{fig:regionpuzzle}Assembling optimality regions for the
Bradley--Terry model.}
\end{figure}

\begin{remark}\label{l:amoeba}
  Figures~\ref{fig:65pointRegion} and~\ref{fig:redblue} are
  reminiscent of the amoebas in tropical geometry.  It would be
  interesting to investigate, if the logarithmic algebraic geometry
  that arises in $\beta$-space from the polynomial constraints in
  $\lambda$-space offers new insights.
\end{remark}

\section{Discussion}
\label{sec:discussion}

This paper explains the parameter regions of optimality for
experimental design of the Bradley--Terry model, with the strongest
results for 4 alternatives.  In practical applications this knowledge
can be put to use as follows: First, with a screening experiment,
initial knowledge of approximate parameters is attained.  The initial
guess lies in one of the full-dimensional regions illustrated in
Figure~\ref{fig:regionpuzzle}.  Depending on which region it is, one
can use specific knowledge about the optimal design weights~$w_{ij}$.
For example, there are explicit polynomial formulas for how the
optimal weights depend on the location in parameter space.
Section~\ref{s:fouralternatives} contains explicit such formulas for
the case of 4 alternatives.

In the case that a screening experiment reveals parameters in a region
where saturated designs are optimal, the solution becomes particularly
pleasant: One only needs to assign equal weights to $m-1$ of the
pairs.  The characterization of regions of optimality of saturated
designs is complete, for any number of alternatives
(Theorem~\ref{t:regions}).

We illustrate the effect of choosing the right design by computing the
efficiency of the uniform design (assigning equal weights to all
pairs) in the case of four alternatives.  Consider the line in
parameter space that is specified by $ 2 \beta_{2} = \beta_1 $,
$4\beta_{3} = 5\beta_{1}$.
\begin{figure}[ht]
  \centering \includegraphics{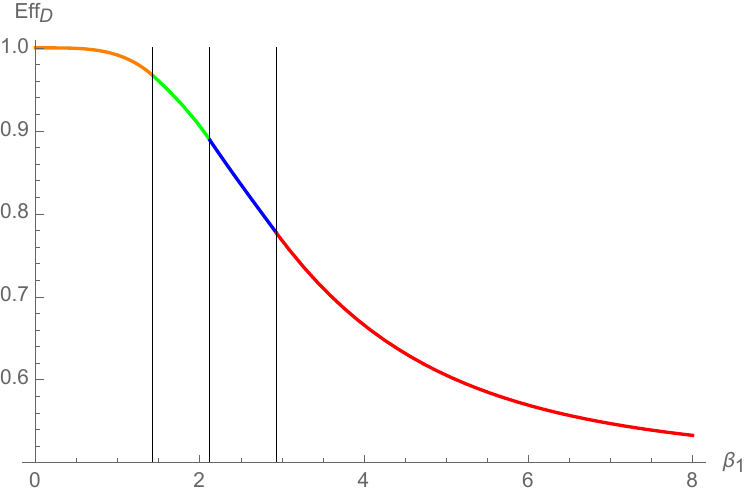}
  \caption{Efficiency of the uniform design along a line in
    $\beta$-space. \label{fig:efficiency}}
\end{figure}
Figure~\ref{fig:efficiency} shows the efficiency of the uniform design
along that line.  At
$\beta = (\beta_{1},\beta_{2},\beta_{3}) = (0,0,0)$ the uniform design
is optimal. As $\beta$ grows, the efficiency decreases.  First
the weights should be adjusted and starting at approximately $1.4$ a
$5$-point design would be optimal.  Around $2.1$ a $4$-point design
becomes optimal and finally, from $2.9$, a saturated design is
optimal.  Clearly, working with a uniform design in the case that the
support should be smaller is inefficient.  In the limit
$\beta \to \infty$ the uniform design requires twice as many
observations as the optimal saturated design.

We outline some further research directions now.  For full support
designs, by Corollary~\ref{c:silvey} the region of optimality is given
by the equations
\[
\lambda_{ij}f(i,j)^T M(\xi^*,\beta)^{-1} f(i,j) = (m-1)
\]
and positivity constraints $\lambda_{ij} > 0$. We hope that tools from
real algebraic geometry can shed further light on such semi-algebraic
sets, especially for designs with full support, as their
semi-algebraic sets contain no complicated inequalities.

The Bradley--Terry model considered here is only for the $m$ levels of
one attribute and an extension to more attributes is conceivable.  The
computational challenges of finding optimal designs are formidable and
a nice geometry as in the present case is not expected.

 In the case of optimality, the equations above express the weights
 of~$\xi$ in terms of the parameters.  We conjecture that the equations
 can be solved in the following sense.
 \begin{conj}
   The $ \binom{m}{2} $ weights of a fully supported $D$-optimal design
   are rational functions in the intensities and of numerator
   degree~$\binom{m}{2} + m -1$.
 \end{conj}
 An example of such expressions are the degree 9 equations in
 Section~\ref{s:fourpointfullsupport}.

 \begin{remark}\label{rmk:structure}
 The solution for the four-dimensional case reveals that the numerator
 of a weight $ w_{ij} $ is a sum of 10 monomials. These monomials can
 be described combinatorially as follows. For simplicity, let $i=1$ and
 $j=2$. Then 8 of the 10 monomials are products of the squarefree
 monomial
 $
 \lambda_{12}\lambda_{13}\lambda_{14}\lambda_{23}\lambda_{24}\lambda_{34}
 $ with monomials of the form $ \lambda_{ij}\lambda_{ik}\lambda_{kl} $,
 where $ (ij,ik,kl) $ are edges of the $ 8 $ graphs that are either
 paths or trees on four vertices and that do not contain the
 edge~$ (1,2) $. Furthermore, the monomials that come from a graph with
 a node of degree $ 3 $ have a positive sign, while the monomials from
 paths have a negative sign. The remaining two monomials do not show
 such an easy structure and it remains open, why they are of the form
 $
 \lambda_{13}^2\lambda_{14}^2\lambda_{23}^2\lambda_{24}^2(\lambda_{12}+2\lambda_{34})
 $.  The complete design is generated by permutations acting on the
 indices of the numerator described above, while the denominator of the
 weights is just the sum of all the numerators, that is, a normalization.
 \end{remark}

 From the structure in the case of $4$ alternatives, one can at least
 partially conjecture the structure of a solution in higher dimensions.
 In the case of $5$ alternatives, we conjecture that for full support
 designs the function that expresses $w_{ij}$ in the intensities
 $\lambda_{ij}$ satisfies the following rules: It is of the form of a
 polynomial divided by a normalization.  The numerator polynomial is of
 degree $\binom{m}{2} + m -1$ (i.e.~14 for $m=5$) and composed as
 follows. Start with the monomial
 $\lambda_{12}\lambda_{13} \cdots \lambda_{m-1,m}$.  To construct the
 weight for the comparison $ (1,2) $, multiply it with a square-free
 product of $m-1$ of the variables $ \lambda_{ij} $, where $ ij $ is an
 edge in a spanning tree on $[m]$ which does not contain $ (1,2) $.  Sum
 these monomials over all trees that do not contain $ (1,2) $.  For $ n=5 $,
 only 50 out of the 125 trees qualify.  In this summation, trees of
 maximal degree 2 receive a negative sign, the others a positive sign.
 Additionally, we may have to add monomials of a still unknown structure
 as in Remark~\ref{rmk:structure} above. We expect a similar structure
 in the denominator for $ 5 $ alternatives as for four, so that there
 is a sum of monomials in the denominator that is multiplied by
 $ 4 $. As there are $ 125 $ trees, this would make $ 500 $ monomials
 from the tree-structure. This coincides with having $ 50 $ monomials
 from trees in the numerator, as there are $ 10 $ weights for $ 5 $
 alternatives.  In comparison, for $4$ alternatives, there are
 $3\cdot 22=66 $ monomials in the denominator, but only $6\cdot 8=48 $
 come from the described graph structure.  The implications of these
 observations are still unknown.

\subsection*{Acknowledgement}
The authors are supported by the Deutsche Forschungsgemeinschaft DFG
under grant 314838170, GRK 2297 MathCoRe.

\bibliographystyle{alpha}
\bibliography{bibliography}

\newcommand{\etalchar}[1]{$^{#1}$}
\begin{thebibliography}{PAV{\etalchar{+}}13}

\bibitem[BCR13]{Bochnak}
J.~Bochnak, M.~Coste, and M.~Roy.
\newblock {\em Real algebraic geometry}, volume~36.
\newblock Springer, New York, 2013.

\bibitem[BMS04]{Berman2004}
D.~R. Berman, S.~C. McLaurin, and D.~D. Smith.
\newblock Ranking whist players.
\newblock {\em Discrete Math.}, 283(1-3):15--28, 2004.

\bibitem[BT52]{BradleyTerry}
R.~A. Bradley and M.~E. Terry.
\newblock Rank analysis of incomplete block designs. {I}. {T}he method of
  paired comparisons.
\newblock {\em Biometrika}, 39:324--345, 1952.

\bibitem[Che53]{Chernoff}
H.~Chernoff.
\newblock Locally optimal designs for estimating parameters.
\newblock {\em Ann. Math. Statistics}, 24:586--602, 1953.

\bibitem[CMP07]{Callaghan2007}
T.~Callaghan, P.~J. Mucha, and M.~A. Porter.
\newblock Random walker ranking for {NCAA} division {I}-{A} football.
\newblock {\em Amer. Math. Monthly}, 114(9):761--777, 2007.

\bibitem[DMJ13]{DuchiMackeyJordan2013}
J.~C. Duchi, L.~Mackey, and M.~I. Jordan.
\newblock The asymptotics of ranking algorithms.
\newblock {\em Ann. Statist.}, 41(5):2292--2323, 2013.

\bibitem[DMP04]{dette2004optimal}
H.~Dette, V.~B Melas, and A.~Pepelyshev.
\newblock Optimal designs for a class of nonlinear regression models.
\newblock {\em The Annals of Statistics}, 32(5):2142--2167, 2004.

\bibitem[Fec66]{Fechner1860}
G.~T. Fechner.
\newblock {\em Elemente der Psychophysik (1860). English translation: Howes,
  D.H., Boring, E.C. (Eds.) and Adler, H.E. (transl.), Elements of
  Psychophysics.}
\newblock Holt, Rinehart and Winston New York, 1966.

\bibitem[Gab00]{gabrielsen2000paired}
G.~Gabrielsen.
\newblock Paired comparisons and designed experiments.
\newblock {\em Food Quality and Preference}, 11(1-2):55--61, 2000.

\bibitem[GRF03]{Graves2003}
T.~Graves, C.~S. Reese, and M.~Fitzgerald.
\newblock Hierarchical models for permutations: analysis of auto racing
  results.
\newblock {\em J. Amer. Statist. Assoc.}, 98(462):282--291, 2003.

\bibitem[GS]{M2}
D.~R. Grayson and M.~E. Stillman.
\newblock {M}acaulay2, a software system for research in algebraic geometry.
\newblock Available at \url{http://www.math.uiuc.edu/Macaulay2/}.

\bibitem[GS08]{Grasshoff2008}
U.~Gra{\ss}hoff and R.~Schwabe.
\newblock Optimal design for the {B}radley--{T}erry paired comparison model.
\newblock {\em Statistical Methods and Applications}, 17(3):275--289, 2008.

\bibitem[HT98]{Hastie1998}
T.~Hastie and R.~Tibshirani.
\newblock Classification by pairwise coupling.
\newblock {\em Ann. Statist.}, 26(2):451--471, 1998.

\bibitem[Hun04]{Hunter2004}
D.~R. Hunter.
\newblock {MM} algorithms for generalized {B}radley--{T}erry models.
\newblock {\em Ann. Statist.}, 32(1):384--406, 2004.

\bibitem[HYTC20]{han2020asymptotic}
R.~Han, R.~Ye, C.~Tan, and K.~Chen.
\newblock Asymptotic theory of sparse {Bradley--Terry} model.
\newblock {\em Annals of Applied Probability}, 30(5):2491--2515, 2020.

\bibitem[KKT06]{Kobayashi2006}
K.~Kobayashi, H.~Kawasaki, and A.~Takemura.
\newblock Parallel matching for ranking all teams in a tournament.
\newblock {\em Adv. in Appl. Probab.}, 38(3):804--826, 2006.

\bibitem[KOS16]{kahle2015algebraic}
T.~Kahle, K.~Oelbermann, and R.~Schwabe.
\newblock Algebraic geometry of {P}oisson regression.
\newblock {\em Journal of Algebraic Statistics}, 7:29--44, 2016.

\bibitem[KRS20]{kahle20:DMV}
T.~Kahle, F.~R\"ottger, and R.~Schwabe.
\newblock Geometrie optimaler {V}ersuchspl\"ane.
\newblock {\em DMV Mitteilungen}, 20(2):71--76, 2020.

\bibitem[PAV{\etalchar{+}}13]{sostools}
A.~Papachristodoulou, J.~Anderson, G.~Valmorbida, S.~Prajna, P.~Seiler, and
  P.~A. Parrilo.
\newblock {\em {SOSTOOLS}: Sum of squares optimization toolbox for {MATLAB}}.
\newblock \url{https://arxiv.org/abs/1310.4716}, 2013.
\newblock Available from \url{http://www.cds.caltech.edu/sostools}.

\bibitem[Puk06]{pukelsheim2006optimal}
F.~Pukelsheim.
\newblock {\em Optimal design of experiments}.
\newblock Classics in applied mathematics. Society for Industrial and Applied
  Mathematics, 2006.

\bibitem[RS16]{Radloff2016}
M.~Radloff and R.~Schwabe.
\newblock Invariance and equivariance in experimental design for nonlinear
  models.
\newblock In J.~Kunert, C.~H. M{\"u}ller, and A.~C. Atkinson, editors, {\em
  mODa 11 - Advances in Model-Oriented Design and Analysis Proceedings}, pages
  217--224. Springer International Publishing, Cham, 2016.

\bibitem[Sil80]{silvey1980optimal}
S.D. Silvey.
\newblock {\em Optimal design: an introduction to the theory for parameter
  estimation}.
\newblock Monographs on applied probability and statistics. Chapman and Hall,
  1980.

\bibitem[SS89]{shah1989}
K.~R. Shah and B.~K. Sinha.
\newblock {\em Theory of optimal designs}, volume~54 of {\em Lecture Notes in
  Statistics}.
\newblock Springer-Verlag, New York, 1989.

\bibitem[SW12]{sturmfelswelker}
B.~Sturmfels and V.~Welker.
\newblock {Commutative algebra of statistical ranking}.
\newblock {\em J. Algebra}, 361:264--286, 2012.

\bibitem[SY99]{SimonsYao1999}
G.~Simons and Y.~Yao.
\newblock Asymptotics when the number of parameters tends to infinity in the
  {B}radley-{T}erry model for paired comparisons.
\newblock {\em Ann. Statist.}, 27(3):1041--1060, 1999.

\bibitem[Tor04]{Torsney2004}
B.~Torsney.
\newblock Fitting {B}radley {T}erry models using a multiplicative algorithm.
\newblock In J.~Antoch, editor, {\em COMPSTAT 2004 --- Proceedings in
  Computational Statistics}, pages 513--526, Heidelberg, 2004. Physica-Verlag
  HD.

\bibitem[Zer29]{zermelo29}
E.~Zermelo.
\newblock Die {B}erechnung der {Turnier-Ergebnisse} als ein {M}aximumproblem
  der {W}ahrscheinlichkeitsrechnung.
\newblock {\em Mathematische Zeitschrift}, 29:436--460, 1929.

\end{thebibliography}

\end{document}